\documentclass[letterpaper,10pt]{amsart}
\usepackage{thmtools}
\usepackage{amssymb}
\usepackage{xcolor}
\usepackage{hyperref}
\usepackage{mathrsfs} 
\usepackage{enumerate} 
\usepackage{dsfont}

\usepackage[style=alphabetic,sorting=nty,backend=bibtex,useprefix]{biblatex}
\bibliography{bibliografia}

\definecolor{carminepink}{rgb}{0.92, 0.3, 0.26}

\hypersetup{backref,colorlinks=true,allcolors=carminepink} 

\declaretheorem[parent=section]{theorem}
\declaretheorem[sibling=theorem]{proposition}
\declaretheorem[sibling=theorem]{lemma}

\declaretheorem[sibling=theorem,style=remark]{remark}

\declaretheorem[sibling=theorem,style=definition]{definition}
\declaretheorem[sibling=theorem,style=condition]{condition}

\DeclareSymbolFont{bbold}{U}{bbold}{m}{n}
\DeclareSymbolFontAlphabet{\mathbbold}{bbold}

\newcommand{\reals}{\mathbb{R}}		
\newcommand{\naturals}{\mathbb{N}}					          
\newcommand{\integers}{\mathbb{Z}}
\newcommand{\Ex}{\mathbb{E}}				              
\newcommand{\Prob}{\mathbb{P}}    
				         
\newcommand{\Tq}{\mathbb{T}^q}

\newcommand{\betac}{\beta_{j,k}}                          
                              
\newcommand{\barg}{\frac{n}{B^j}}                          
\newcommand{\cubew}{\lambda_{j,k}}
\newcommand{\cubep}{\xi_{j,k}}
\providecommand{\needlet}[1]{\psi_{j,k}\left(#1\right)}	
\providecommand{\bfun}[1]{b\left(#1\right)}				  
\newcommand{\sumk}{\sum_{k=1}^{K_j}}   
\newcommand{\ind}[1]{\mathbbold{1}_{\bbra{#1}}}

\providecommand{\ncomp}[1]{n_{#1}}
\providecommand{\thetacomp}[1]{\theta_{#1}}

\definecolor{electricultramarine}{rgb}{0.25, 0.0, 1.0}

\providecommand{\cc}[1]{\overline{#1}}
	
\providecommand{\bra}[1]{\left(#1\right)}			      
\providecommand{\sbra}[1]{\left[#1\right]}			    
\providecommand{\bbra}[1]{\left\{#1\right\}}

\renewcommand{\ll}{\ell}							         
\renewcommand{\l}{\left}                                      
\renewcommand{\r}{\right}

\newcommand{\E}{\mathbb{E}}

\newcommand{\norm}[1]{\left\lVert#1\right\rVert}

\newcommand{\R}{\mathbb{R}}
\newcommand{\cvlaw}{\overset{\mathscr{L}}{\rightarrow}}

\calclayout

\begin{document}

\author{Solesne Bourguin}
\address{Boston University, Department of Mathematics and Statistics, 111 Cummington Mall, Boston, MA 02215, USA}
\email{solesne.bourguin@gmail.com}
\author{Claudio Durastanti}
\address{Ruhr University Bochum,  Faculty of Mathematics,  D-44780 Bochum, Germany}
\email{claudio.durastanti@gmail.com}
\thanks{C. Durastanti is supported by the German DFG grant \textit{GRK} 2131}
  
\title[On normal approximations for the two-sample problem]{On normal approximations for the two-sample problem on multidimensional tori} 
\begin{abstract}
In this paper, quantitative central limit theorems for $U$-statistics on the $q$-dimensional torus defined in the framework of the two-sample problem for Poisson processes are derived. In particular, the $U$-statistics are built over tight frames defined by wavelets, named toroidal needlets, enjoying excellent localization properties in both harmonic and frequency domains. The Berry-Ess\'een type bounds associated with the normal approximations for these statistics are obtained by means of the so-called Stein-Malliavin techniques on the Poisson space, as introduced by Peccati, Sol\'e, Taqqu, Utzet (2011) and further developed by Peccati, Zheng (2010) and Bourguin, Peccati (2014). Particular cases of the proposed framework allow to consider the two-sample problem on the circle as well as the local two-sample problem on $\R^q$ through a local homeomorphism argument.
\end{abstract}
\subjclass[2010]{60F05, 60G55, 62G09, 62G10, 62G20.}
\keywords{Two-sample problem; $U$-statistics; Poisson processes; wavelets; needlets; harmonic analysis on the torus, Stein-Malliavin method.}
\maketitle

\section{Introduction}\label{sec:intro}
\noindent The aim of this paper is to establish quantitative central limit theorems, by means of Stein-Malliavin techniques, for wavelet-based U-statistics on $q$-dimensional tori arising in the context of the two-sample problem for Poisson processes. The two-sample (or homogeneity) problem for Poisson processes can be described as follows: let $N_{1}$ and $N_{2}$ denote two independent Poisson processes observed on a measurable space $\mathbb{X}$, whose intensities with respect to some positive non-atomic $\sigma$-finite measure $\mu$ are denoted by $f_{1}$ and $f_{2}$, respectively. Given the observation of $N_{1}$ and $N_{2}$, the two sample problem aims at testing the null-hypothesis $(H_{0}) \colon f_{1}=f_{2}$ versus the alternative hypothesis $(H_1) \colon f_{1}\neq f_{2}$, see for instance \cite{fromont} for an in-depth description. In such a problem, two-sample $U$-statistics arise very naturally (see for instance \cite{fromont, vandervaart, dasgupta}) as they can be used to approximate both the distribution under the null as well as the alternative hypotheses (note that in the case of the alternative hypothesis, one has to deal with different underlying distributions for the Poisson processes, see e.g. \cite{lee}). 
\\~\\
This paper assumes a slightly different, but equivalent framework in which the null-hypothesis $(H_{0})$ is that observations of a unique Poisson process $N$, sampled over two disjoint $q$-dimensional tori $\mathbb{T}^q_1$ and $\mathbb{T}^q_2$, are distributed according to the same intensity with respect to $\mu$ (note when working under the null-hypothesis, this is equivalent to considering two independent Poisson processes over the same support). For this purpose, let $\left\lbrace N_t \colon t \geq 0\right\rbrace $ be a Poisson process over a $q$-dimensional torus $\mathbb{T}^q$ with control measure given by 
\begin{equation}
\label{eqn:control}
\mu_t(d\theta) :=R_t f\left( \theta \right)  d\theta,
\end{equation}
where $R_t>0$ denotes, roughly speaking, the expected number of observations at time $t>0$ and $f$ is a density function over $\mathbb{T}^q$ satisfying one of two possible mild regularity conditions on its spectral decomposition (see Conditions \ref{cond1} and \ref{cond2} in Subsection \ref{sub:approx}). Consider the estimator, given in the form of a $U$-statistic, defined by
\begin{eqnarray}
\label{mainestimator}
U_{j}\left( t\right)  &:=& \sum_{k=1}^{K_j}\left[ \left( \int_{\mathbb{T}_{1}^{q}}\psi _{j,k}\left( \theta \right) N_{t}\left( d\theta \right) -\int_{\mathbb{T}_{2}^{q}}\psi _{j,k}\left( \theta \right) N_{t}\left( d\theta\right) \right) ^{2} \right.  \nonumber \\
&& \left. \qquad\qquad\qquad\qquad\qquad\qquad\qquad\qquad - \int_{\mathbb{T}_{1}^{q}}\psi _{j,k}^{2}\left( \theta \right)
N_{t}\left( d\theta \right) -\int_{\mathbb{T}_{2}^{q}}\psi _{j,k}^{2}\left(
\theta \right) N_{t}\left( d\theta \right) \right], 
\end{eqnarray}
where, given a scale parameter $B$, the set $\left\lbrace \psi_{j,k} \colon j\geq0,\ k=1,\ldots,K_j \right\rbrace $ is the set of the $q$-dimensional toroidal needlets for which the index $j$ denotes the multiresolution level, while $K_j$ stands for the cardinality of needlets at a given resolution level $j$ is fixed (a more detailed description can be found in Subsection \ref{sub:needlets}).
\\~\\
The strategy for deriving the upcoming quantitative normal approximation results for the two-sample problem on multidimensional tori will be to make use of the celebrated Stein-Malliavin techniques obtained in \cite{PSTU} through the combination of Stein's method with the Malliavin calculus of variations for Poisson functionals (see Section \ref{sub:approx} for a self-contained description of these techniques). 
\\~\\
The Berry-Ess\'een type bounds associated with the limit theorems proved below only depend on the resolution level $j$ of the needlets, on the expected number of observations $R_t$ sampled at time $t$, and on the spectral parameter $\alpha$ controlling the rate of decay of the density function $f$. As pointed out in Remark \ref{remarkonregularity}, these bounds, through their dependence on the regularity parameter $\alpha$ of the density function $f$ appearing in the control measure of the Poisson process provides a new quantitative estimation of the impact of the regularity of $f$ in the quality of the normal approximation of the estimator \eqref{mainestimator} (as well as the impact on the rate of convergence to the Gaussian distribution when such a convergence applies). A rigorous formulation of these quantitative bounds, as well as of the role of the regularity of $f$, is given in Theorem \ref{maintheorem}, and are the main findings of this paper, and to our knowledge the first quantitative bounds for the two-sample problem involving the regularity parameter of the density function $f$.
 
\subsection{An overview in the literature}\label{sub:overview}
\noindent The comparison between two probability distributions has been an important and long-lasting subject with applications in a wide range of fields, such as biology, medicine, physics and cosmology (among a lot of others). Following the seminal papers \cite{cox, przy}, dealing with homogeneous Poisson processes, two-sample test statistics were introduced within many different settings: some kernel-based procedures were proposed in \cite{aht, gretton, ht2002}, while the study of asymptotic properties of U-statistics based on non-homogeneous Poisson processes was addressed, for instance, in \cite{deshpande, fromont}. The reader is also referred to \cite{lee,vandervaart} for further details and discussions on this topic.
\\~\\
Stein-Malliavin techniques, first introduced in \cite{nourdinpeccati} have become increasingly popular within the scientific community. Initially used to establish Berry-Ess\'een type bounds for functionals of Gaussian fields, these techniques have been extended to framework of Poisson random measures (see e.g. \cite{BouPec,PSTU, PecZheng}). More recently, such methods have been generalized to the setting of spectral theory of general Markov diffusion generators (see e.g. \cite{simon,ledoux}). The reader is reffered to the text \cite{noupebook} for a detailed introduction to this topic. In the context of statistical analysis over the sphere, Stein-Malliavin methods have proven quite powerful to establish asymptotic normality of wavelet-based linear and nonlinear statistics (respectively in \cite{durastanti4, dmp} and \cite{bdmp}) built over the sphere by means of the so-called spherical needlets.
\\~\\
Spherical needlets were introduced in \cite{npw2,npw1} and form a tight frame on the sphere, so that a reconstruction formula (the counterpart of the harmonic expansion in the wavelet framework) holds. Furthermore, they enjoy remarkable localization properties in both harmonic and spatial domains, while their statistical properties when applied to the study of spherical random fields were investigated in \cite{bkmpAoS}. Further remarkable statistical applications can be found, for instance, in \cite{bkmpAoSb, cammar,dlmejs}. This paper makes use of wavelets on $\mathbb{T}^q$, built exactly as in \cite{npw1}, called toroidal needlets (see e.g. Subsection \ref{sub:needlets} for technical details). Other extensions on various manifolds such as the unit ball of $\mathbb{R}^3$ and spherical spin fiber bundles can be found in \cite{dfhmp} and \cite{gelmar}, respectively. The monograph \cite{MaPeCUP} should be noted to be a remarkable go-to reference for details and discussions on both the theoretical and applied aspects of wavelets and needlets.
\\~\\
Finally, note that the choice of $\Tq$ as the support of the Poisson processes under consideration in this work is very natural in view of the size of the literature dealing with statistical tests aimed at comparing two distributions over circular data, which corresponds the the case $q=1$ of the framework of thhis paper (see for instance \cite{eplett,mardiajupp,jamma} among many others). In this light, the framework of a $q$-dimensional tori $\Tq$ can be viewed as a unifying generalization as it also encompasses the local two-sample problem over $\R^q$ through the fact that $\R^q$ and $\Tq$ are locally homeomorphic and the spatial concentration of the toroidal needlets which ensures consistency of the normal approximation results for any local approximation of $\R^q$ by $\Tq$. 
      
\subsection{Main results}\label{sub:approx}
\noindent The Poisson random measures considered in the framework of this paper are Poisson random measures over $\mathbb{T}^q$ with control measures given by \eqref{eqn:control}. The density function $f$ can be viewed in terms of its harmonic expansion (see Section \ref{sec:background} for more details), that is, for any $\theta \in \mathbb{T}^q$, it holds that 
\begin{equation}\label{eqn:densityfunc}
f\l(\theta\r)=\sum_{n \in \mathbb{Z}^q} a_n s_n\l(\theta\r),
\end{equation}
where $\bbra{s_n \colon n \in \mathbb{Z}^q}$ is the set of eigenfunctions of the Laplace-Beltrami operator on $\mathbb{T}^q$, and forms an orthonormal system for $\Tq$.  $\bbra{a_n \colon n \in \mathbb{Z}^q}$ are the corresponding Fourier (or harmonic) coefficients. Two sets of assumptions on the density function $f$ (more precisely on the harmonic expansion of $f$) will be used throughout the paper, namely
\begin{condition}\label{cond1}For any $n \in \integers^q$, there exist two constants $c>0$ and $\alpha \geq \frac{1}{2}$ such that 
\begin{equation}\label{eqn:condition1}
a_{n}=c\left( \ell_n+1\right) ^{-\alpha },
\end{equation}
where $\ell_n$ is the eigenfunction of the Laplace-Beltrami operator associated with the multiindex $n$.
\end{condition}
\begin{condition}\label{cond2}For any $n \in \integers^q$, there exist two constants $c>0$ and $\alpha \geq \frac{1}{2}$ such that 
	\begin{equation}\label{eqn:condition2}
	a_{n}=c\prod_{m=1}^q\left( n_{m}+1\right) ^{-\alpha },
\end{equation}
where $n_{m}$, $m=1,\ldots,q$, are the components of the multiindex $n$. 
\end{condition}
\begin{remark} Conditions \ref{cond1} and \ref{cond2} imply a strong interdependence and statistical independence between the components of the coefficients of the density function associated to the frequency $n$, respectively. The parameter $\alpha$ controls the rate of decay of the density function $f$. Note that the two conditions reduce to the same one if $q=1$. 
\end{remark}
\noindent Define  $\widetilde{U}_{j}\left( t\right)$ to be the normalized version of the $U$-statistic defined in \eqref{mainestimator}, namely
\begin{equation}
\label{normalizeduj}
\widetilde{U}_{j}\left( t\right) := \frac{U_{j}\left( t\right)}{\sqrt{Var\left( U_{j}\left( t\right)\right) }}.
\end{equation}
The following theorem is the main result of this paper, providing quantitative normal approximation bounds for the two-sample problem over $q$-dimensional tori. Throughout the paper, the symbol $\cvlaw$ denotes convergence in law, while the symbol $\sim$ denotes an asymptotic equivalent and the symbols $\lesssim$ and $\gtrsim$ stand for asymptotically equivalent upper and lower bounds, respectively.
\begin{theorem}
	\label{maintheorem}Let $\widetilde{U}_{j}\left( t\right)$ be given by \eqref{normalizeduj}, $Z \sim\mathcal{N}\left( 0,1\right)$ be a standard Gaussian random variable and $d_W$ denote the Wasserstein metric (see upcoming Definition \ref{defwassersteinmetric}).
\begin{enumerate}[(i)]
\item Assume that Condition \ref{cond1} prevails. Then, it holds that 
	\begin{equation*}
	d_{W}\left( \widetilde{U}_{j}\left( t\right),Z\right)
	\lesssim \begin{cases}  R_t^{-\frac{1}{2}}B^{\frac{j}{4}}+R_t^{-\frac{1}{2}}+B^{-\frac{j}{2}} & \text{if } \alpha \in \left[  1,+\infty \right),  \\  R_t^{-\frac{1}{2}}B^{\frac{j}{4}}+R_t^{-\frac{1}{2}}B^{j\left(1-\alpha\right)}+B^{j\bra{1-2\alpha}} & \text{if } \alpha \in \left[ \frac{1}{2},1\right).  \end{cases}
	\end{equation*}
Furthermore, if $R_t^{-\frac{1}{2}}B^{j \max\left(\frac{1}{4},1-\alpha \right)} \rightarrow 0$ as $t \rightarrow \infty$, then
	\begin{equation*}
	\widetilde{U}_{j}\left( t\right) \cvlaw \mathcal{N}\left( 0,1\right).
	\end{equation*}
\item Assume that Condition \ref{cond2} prevails. Then, it holds that 
	\begin{equation*}
	d_{W}\left( \widetilde{U}_{j}\left( t\right) ,Z \right)
	\lesssim \begin{cases}  R_t^{-\frac{1}{2}}B^{\frac{jq}{4}}+R_t^{-\frac{1}{2}}+B^{-\frac{jq}{2}} & \text{if } \alpha \in \left[ 1,+\infty \right), \\ R_t^{-\frac{1}{2}}B^{\frac{jq}{4}}+R_t^{-\frac{1}{2}}B^{jq\left(1-\alpha\right)}+B^{jq\bra{1-2\alpha}} & \text{if } \alpha \in \left[ \frac{1}{2},1\right).  \end{cases}
	\end{equation*}%
Furthermore, if $ R_t^{-\frac{1}{2}}B^{jq\max\left(\frac{1}{4},1-\alpha \right)} \rightarrow 0$ as $t \rightarrow \infty$, then
	\begin{equation*}
	\widetilde{U}_{j}\left( t\right) \cvlaw \mathcal{N}\left( 0,1\right).
	\end{equation*}
\end{enumerate}
\end{theorem}
\begin{remark}
\label{remarkonregularity}
Observe that the above results indicate a regime transition whenever the regularity parameter $\alpha$ of the density function $f$ hits the value $\alpha =1$. For $\alpha < 1$, the rates of convergence depends on $\alpha$ and are slower the closer $\alpha$ is to $1/2$, whereas for $\alpha \geq 1$, the regularity of $f$ does not influence the rates of convergence anymore. This phenomenon provides a new quantitative interpretation of the impact of the regularity of $f$ in the normal convergence of the studied estimator.
\end{remark}
\subsection{Plan of the paper}\label{sub:plan}
\noindent The rest of the paper is organized as follows: Section \ref{sec:background} contains the needed elements of harmonic and wavelet analysis on the $q$-dimensional torus as well as the version of the Stein-Malliavin approximation results for Poisson random measures used here. Section \ref{sec:approx} contains the proof of the main result, namely Theorem \ref{maintheorem}, while Section \ref{sec:proof} gathers technical auxiliary results along with their proofs.
\section{Background and notation}\label{sec:background}
\noindent This section provides some background on Poisson random measures, quantitative estimates for normal approximations of Poisson multiple integrals, elements of Fourier analysis on $\mathbb{T}^q$ (see e.g. \cite{grafokos}) and the construction of toroidal needlets and their properties (see e.g. \cite{npw2,npw1}). From now on, $n \in\integers^q$ will denote the vector $n=\l(n_1,\ldots,n_q\r)$ and $\ell_n:=\norm{n}_{\ell^2\left(\mathbb{Z}^q\right)}=\sqrt{\sum_{i=1}^q \ncomp{i}^2}$ will denote the eigenvalue of the Laplace-Beltrami operator associated with the multiindex $n$.
\subsection{Poison random measures and multiple Wiener-It\^o integrals}\label{poissonrmmultint}
Let $\bra{\Omega,\mathcal{F},\Prob}$ be a probability space, $\left(\mathbb{X},\mathcal{X} \right) $ be a polish space and $\mu$ be a positive, $\sigma$-finite measure over $\left(\mathbb{X},\mathcal{X} \right) $ with no atoms. Denote by $\mathcal{X}_{\mu}$ the class of those $A \in \mathcal{X}$ such that $\mu(A) < \infty$. Then, a Poisson random measure with control $\mu$ on $\left(\mathbb{X},\mathcal{X} \right) $ with values in $\bra{\Omega,\mathcal{F},\Prob}$ is a map $N \colon \mathcal{X}_{\mu} \rightarrow \Omega$ with the following properties:
\begin{enumerate}[(1)]
\item For any set $A$ in $\mathcal{X}_{\mu}$, $N(A)$ is a Poisson random variable with parameter $\mu(A)$;
\item If $r \in \mathbb{N}$ and $A_1, \ldots ,A_r \in \mathcal{X}_{\mu}$ are disjoint, then $N\left( A_1\right) , \ldots ,N\left( A_r\right)$ are independent;
\item If $r \in \mathbb{N}$ and $A_1, \ldots ,A_r \in \mathcal{X}_{\mu}$ are disjoint, then $N\left(\bigcup_{i=1}^{r}A_i \right) = \sum_{i=1}^{r}N\left( A_i\right)  $.
\end{enumerate}
\noindent If $N$ is a Poisson random measure on $\left(\mathbb{X},\mathcal{X} \right) $ with control $\mu$, denote by $\hat{N}$ the associated centered Poisson random measure
\begin{equation*}  
\hat{N} (A) = N(A) - \mu(A), \ \ A \in \mathcal{X}_{\mu}.
\end{equation*}
Note that one can regard Poisson random measures as acting on $L^2\left(\mu \right) $ through the identifications $N\left( \mathds{1}_{A}\right)   = N(A)$ and $\hat{N} \left( \mathds{1}_{A}\right)  = N(A) - \mu(A)$, for all $A \in \mathcal{X}_{\mu}$.
\noindent For every deterministic function $h \in L^2\left( \mu\right) $, write
\begin{equation*}  
I_{1}\left( h\right) = \hat{N}(h) = \int_{\mathbb{X}}h(x)\hat{N}(dx)
\end{equation*}
to indicate the Wiener-It\^o integral of $h$ with respect to $\hat{N}$. For every $q \geq 2$ and every symmetric function $f \in L^2\left(\mu^q \right) $, denote by $I_{q}\left(f\right) $ the multiple Wiener-It\^o integral of order $q$ of $f$ with respect to $\hat{N}$. For any real constant $c$, set $I_{0}(c) = c$ and for any $f \in L^2\left(\mu^q \right) $ (not necessarily symmetric), set $I_{q}(f) = I_{q}(\widetilde{f} ) $, where 
\begin{equation*}
\widetilde{f}\left( x_1 , \ldots ,x_q\right)  = \frac{1}{q!}\sum_{\pi \in \frak{S}_q}f\left( x_{\pi(1)} , \ldots ,x_{\pi(q)}\right) 
\end{equation*}
denotes the symmetrization of $f$. The reader is referred for instance to \cite[Chapter 5]{peccatitaqqu} for a complete discussion of multiple Poisson integrals and their properties. 
\begin{proposition}
\label{isometryprop}
The following equalities hold for every $q,p \geq 1$ and every symmetric functions $f \in L^2\left(\mu^{q} \right) $ and $g \in L^2\left(\mu^{p} \right) $:
\begin{enumerate}[(1)]
\item $\E\left( I_{q}(f)\right) =0 $;
\item$\E\left( I_{q}(f)I_{p}(g)\right) = q! \left\langle f,g \right\rangle_{L^2\left(\mu^{q} \right) } \mathds{1}_{\left\lbrace q= m\right\rbrace } $.
\end{enumerate}
\end{proposition}
\noindent The following definition introduces the star-contraction operation between two symmetric functions $f \in L^2\bra{\mu^q}$ and $g \in L^2\bra{\mu^p}$.
\begin{definition}
Let $q, p \geq 1$ and let $f\in L^2\bra{\mu^q}$ and $g\in L^2\bra{\mu^p}$ be symmetric functions. The star contraction of index $\bra{r,\ll}$ between $f$ and $g$, denoted by $f \star_r^{\ll} g$, is defined as the $L^2\bra{\mu^{q+p -r -\ell}}$ function given by
	\begin{eqnarray*}
&& f \star^{\ll}_r g\bra{s_1,\ldots,s_{q-r},t_1,\ldots,t_{p-r},\gamma_1,\ldots,\gamma_{r-\ll}}:=\int_{\mathbb{X}^\ll} f\bra{s_1,\ldots,s_{q-r},\gamma_1,\ldots,\gamma_{r-\ll},z_1,\ldots,z_\ll}\\
	&& \qquad\qquad\qquad\qquad\qquad\qquad\qquad\qquad\qquad g\bra{t_1,\ldots,t_{p-r},\gamma_1,\ldots,\gamma_{r-\ll},z_1,\ldots,z_\ll}\mu^{\otimes\ll}\bra{dz_1,\ldots,dz_\ll}.
	\end{eqnarray*}
In particular, $f \star_{0}^{0} g = f \otimes g$, $f \star^{q}_{q} f = \norm{f}^{2}_{L^2\left(\mu^q\right) }$ and $f \star_{q}^{0} f = f^2$.
\end{definition}
\noindent The next statement contains an important product formula for Poisson multiple integrals. Note that the statement involves the star-contraction operators defined above. 
\begin{proposition}
\label{prodformula}
Let $q, p \geq 1$ and let $f\in L^2\bra{\mu^q}$ and $g\in L^2\bra{\mu^p}$ be symmetric functions. Then, 
\begin{equation*}  
I_{q}(f)I_{p}(g) = \sum_{r = 0}^{q \wedge p}r! \binom{q}{r}\binom{p}{r}\sum_{\ell = 0}^{r}\binom{r}{\ell}I_{q+p- r -\ell}\left(\widetilde{f \star^{\ll}_r g} \right).
\end{equation*}
\end{proposition}
~\\

\noindent From now on, let $\mathbb{X} =\reals^+ \times \Tq$ and $\mathcal{X}=\mathcal{B}\bra{\reals^+ \times \Tq}$, where $\mathcal{B}\bra{\reals^+ \times \Tq}$ denotes the class of Borel subsets of $\reals^+ \times \Tq$. $N$ will denote a Poisson random measure on $\reals^+ \times \Tq$ with control measure $\eta=\tau \times \mu$, where $\tau$ is such that $\tau\bra{ds}=R\cdot \lambda\bra{ds}$ with $\lambda\bra{ds}$ denoting the Lebesgue measure on $\reals$ and $R>0$, so that $\tau\bra{\sbra{0,t}}=R \cdot t = R_t$. $\mu$ is defined to be a probability measure on $\Tq$ that is absolutely continuous with respect to the Lebesgue measure on $\Tq$, denoted by $\nu$, so that $$\mu\bra{d\theta}=f\bra{\theta}\nu\bra{d\theta}.$$ For any fixed $t>0$, denote by $N_t$ the Poisson random measure over $\bra{\Tq,\mathcal{B}\bra{\Tq}}$ with control measure $\mu_t := R_t \mu$.
\begin{remark}
As pointed out in \cite{bdmp,dmp,durastanti3}, $t$ can be regarded as the time parameter, so that, for any $A \in \Tq$, $\mu_t\bra{A}$ represents the expected number of observations sampled on $A$ at time $t$.
\end{remark}

\subsection{Stein-Malliavin bounds for Poisson multiple integrals}\label{sub:bounds}

\noindent The Stein-Malliavin bounds that will be made use of in this paper will be stated for the specific Poisson random measure introduced above for notational convenience, but the result holds for any Poisson random measure as defined in Subsection \ref{poissonrmmultint}. Before stating  these bounds, one needs to introduce the so-called Wasserstein metric.
\begin{definition}[Wasserstein metric]
\label{defwassersteinmetric}
Let $\operatorname{Lip}(1)$ denote the class of real-valued Lipschitz functions, from $\R$ to $\R$, with Lipschitz constant less or equal to one, that is functions $h$ that are absolutely continuous and satisfy the relation $\norm{h'}_{\infty} \leq 1$. Given two real-valued random variables $X$ and $Y$, the Wasserstein distance between the laws of $X$ and $Y$, written $d_W\bra{X,Y}$ is defined as 
	\begin{equation*}
	d_W\bra{X,Y}= \sup_{h \in \operatorname{Lip}(1)} \left | \Ex\sbra{h\bra{X}}-\Ex\sbra{h\bra{Y}} \right |. 
	\end{equation*}
\end{definition}
\noindent The following statement is a version of Theorem 4.2 in \cite{PSTU} in the case of double Poisson integrals, stated within the framework of this paper (see also \cite[Example 4.4]{PSTU}).
\begin{proposition}
\label{wassboundfordouble}
Let $Z \sim \mathcal{N}\bra{0,1}$ and assume that $F_j = I_2\bra{h_{j}}$, where the symmetric function $h_{j} \in  L^2\bra{\mu_t^2} $ is such that $2\norm{h_j}^{2}_{L^{2}\left(\mu_t^2 \right) } = 1$, $h_j \star_{2}^{1} h_j \in  L^{2}\left(\mu_t \right) $, $h_j \in L^4\bra{\mu_t^2}$, and
\begin{equation*}
\int_{\Tq}\sqrt{\int_{\Tq}h_j(z,a)^4\mu_t(da)}\mu_t(dz) < \infty.
\end{equation*}
Furthermore, if
\begin{equation*}
\norm{h_j}_{L^4\bra{\mu_t^2}} + \norm{h_j \star^1_1 h_j }_{L^2\bra{\mu_t^2}} + \norm{h_j \star^1_2 h_j }_{L^2\bra{\mu_t}} \rightarrow 0
\end{equation*}
as $j \rightarrow \infty$, then $F_j$ converges in law to $Z$ and the following bound on the Wasserstein distance between $F_j$ and $Z$ holds:
	\begin{equation*}
	d_W\bra{F_j,Z} \leq \sqrt{8} \norm{h_j \star^1_1 h_j }_{L^2\bra{\mu_t^2}} + \left(6 + 2\sqrt{2} \right) \norm{h_j \star^1_2 h_j }_{L^2\bra{\mu_t}} + \sqrt{8} \norm{h_j}_{L^4\bra{\mu_t^2}}^2.
	\end{equation*}
\end{proposition} 
\subsection{Harmonic analysis on multidimensional tori}\label{sub:harmonic}
\noindent As is well-known in the literature (see for instance \cite{grafokos}), the $q$-dimensional torus can be viewed as the direct product of  $q$ unit circles, i.e. $\mathbb{T}^q=\mathbb{S}^{1}\times ...\times \mathbb{S}^{1}\subset \mathbb{C%
}^{q}$. Let the generic coordinates over $\mathbb{T}^{q}$ be given by $\theta =\left(
\thetacomp{1},\ldots,\thetacomp{q}\right) $ and let $%
\nu $ denote the uniform Lebesgue measure over $\mathbb{T}%
^{q}$, that is
\begin{equation}\label{eqn:lebesgue}
\nu \left( d\theta \right) =\prod_{i=1}^{q}\rho \left( d\thetacomp{i}\right) 
\end{equation}%
where $\rho $ is the Lebesgue measure over the unit circle. If we denote $\langle \cdot,\cdot \rangle$ the standard scalar product between $q$ dimensional vectors, the set of functions $s_{n}:\mathbb{T}^q\rightarrow \mathbb{C}$, $%
n=\left( \ncomp{1},...,\ncomp{q}\right) \in \mathbb{Z}%
^{q}$, defined by%
\begin{equation}  \label{eqn:torusbasis}
s_{n}\left( \theta \right) :=\left( 2\pi \right) ^{-\frac{q}{2}%
}\exp \left( i \langle n,\theta \rangle\right) ,
\end{equation}%
 forms an orthonormal basis for the functional space $L^{2}\left( \mathbb{T}^{q},\nu \right)$, see again \cite{grafokos}. 
Indeed, $\mathbb{T}^{1}$ can be also identified as an equivalence class of the quotient space $%
\mathbb{R}/2\pi \mathbb{Z}$: a coordinate system on $\mathbb{T}^{q}$ is therefore provided by the canonical representation in $\left[ 0,2\pi \right) ^{q}$. Furthermore, $\left\{ s_{n} \colon  n\in \mathbb{Z}^{q}\right\} $ also coincides with the set of eigenfunctions associated to $\nabla _{\mathbb{T}^{q}}$, the Laplace-Beltrami operator on $\mathbb{T}^{q}$, given by 
\begin{equation*}
\nabla _{\mathbb{T}^{q}}=\sum_{i=1}^{q}\frac{\partial ^{2}}{\partial x_{i}^{2}},
\end{equation*} 
so that $\left( \nabla _{\mathbb{T}^{q}}+\ell_{n}^{2}\right) s_{n}\left( \theta \right)=0$. Hence, the following orthonormality property holds 
\begin{equation}
\label{eqn:ortho}
\int_{\mathbb{T}^{q}}s_{n_{1}}\left( \theta \right) \overline{s_{n_{2}}}%
\left( \theta \right) \nu \left( d\theta \right) =\delta _{n_{1}}^{n_{2}}.
\end{equation}
Any $f\in L^{2}\left( \mathbb{T}^{q},\nu \right) $ can be represented by its harmonic expansion
\begin{equation}
f\left( \theta \right) =\sum_{n\in \mathbb{Z}^{q}}a_{n}s_{n}\left( \theta
\right),\ \ \theta \in \mathbb{T}^{q},  \label{harmonicexp}
\end{equation}%
where, for all $n \in \integers^q$, the complex-valued coefficients $a_{n}$ are the so-called Fourier
coefficients, given by
\begin{equation}\label{eqn:fouriercoeff}
a_{n}=\int_{\mathbb{T}^{q}}f\left(\theta \right) \overline{s_n}\bra{\theta}\nu \left( d\theta \right).
\end{equation}

\subsection{Toroidal needlets}\label{sub:needlets}
\noindent Let us introduce needlet-like wavelets on the $q$-dimensional torus and describe some of their properties. As already mentioned in Section \ref{sec:intro}, needlets were introduced in the
literature on the $q$-dimensional sphere by Narcowich et al. in
\cite{npw2,npw1} (see also \cite{MaPeCUP} for further details). The construction of an analogous wavelet system on $\mathbb{T}^q$ can be roughly viewed as a natural extension of the standard needlets on $\mathbb{S}^1$ to $\mathbb{T}^q$, so that the technical details of the construction will be omitted here for the sake of brevity. Fix a resolution level $j\in \mathbb{N}$. There exists a set of cubature points $\left\{ \xi _{j,k} \colon k=1, \ldots, K_j\right\} $, $\xi _{j,k}\in 
\mathbb{T}^{q}$, associated to a set of cubature weights $\left\{ \lambda
_{j,k} \colon k=1, \ldots, K_j\right\} $ (see e.g. \cite{npw1}). Roughly speaking, $\mathbb{T}^{q}$ can be represented as a partition of $K_{j}$ subregions, named pixels. Each pixel is centered on the corresponding $\xi _{j,k}$ and its area is given by $\lambda _{j,k}$. Fix a scale parameter $B>1$. Then, the $q$-dimensional toroidal needlets are defined by
\begin{equation*}
\needlet{ \theta} =\sqrt{\cubew}\sum_{n\in \mathbb{Z}%
	^{q}} \bfun{B^{-j} \ell_n}s_{n}\left( \theta \right) \overline{s_{n}}%
\left( \cubep\right).
\end{equation*}%
The so-called window function $b:\mathbb{R\mapsto R}^{+}$ satisfies the following properties:
\begin{enumerate}[(1)]
	\item $b$ has compact support in $\left[ B^{-1},B\right] $;
	\item $b\in C^{\infty }\left( \mathbb{R}\right) $;
	\item the so-called \textit{partition of unity property} holds: for any $c>1$, $	\sum_{j\in \mathbb{N}}b^{2}\left( B^{-j}c\right) =1$.
\end{enumerate}
As a consequence, needlets are characterized by the following pivotal properties. In view of (1), for any $j \in \naturals$, $b\left( B^{-j}\ell_{n}\right) $ is different from zero only over a finite subset of $\mathbb{Z}^{q}$. Let $\Lambda _{j}^{q}=\left\{n\in \mathbb{Z}^{q}:\ell_{n}\in \left[ B^{j-1},B^{j+1}\right] \right\}$, so that it holds that
	\begin{equation}\label{eqn:needlet}
	\needlet{ \theta } =\sqrt{\cubew}\sum_{n\in \Lambda
		_{j}^{q}} \bfun{ B^{-j}\ell_{n}} s_{n}\left( \theta \right) \overline{%
		s_{n}}\left( \cubep\right).
	\end{equation}%
	From (2), toroidal needlets enjoy a quasi-exponential localization property in the
	spatial domain, stated as follows: for any $\theta \in \mathbb{T}^{q}$, $M>0$, there exists $%
	c_{M}>0$ such that 
	\begin{equation*}
	\left\vert \psi _{jk}\left( \theta \right) \right\vert \leq \frac{c_{M}B^{%
			\frac{q}{2}j}}{\left( 1+B^{\frac{q}{2}j}d\left( \theta ,\xi _{jk}\right)
		\right) ^{M}},
	\end{equation*}%
	where $d\left( \theta ,\xi _{j,k}\right) $ is the geodesic distance over $%
	\mathbb{T}^{q}$. Loosely speaking, this property ensures that each needlet $%
	\psi _{j,k}\left( \theta \right) $ is not negligible only if $%
	\theta \in E_{j,k}$. As a consequence, the following bounds on the $L^p$ norms of the toroidal needlets hold (see e.g. \cite{npw2}): for any $p\in\left[\left.1, \infty\right)\right.$, there exist positive constants $c_p, C_p$, depending solely on $p$, such that
	\begin{equation}\label{eqn:Lpnorm}
	c_p B^{jq\bra{\frac{1}{2}-\frac{1}{p}}}\leq \norm{\psi_{j,k}}_{L^p\bra{\mathbb{T}^q,d\nu}} \leq C_p B^{jq\bra{\frac{1}{2}-\frac{1}{p}}}
	\end{equation}
	Finally, from (3), one can infer that the needlet system $\left\{\psi_{j,k} \colon j \geq 0,\ k=1, \ldots, K_j\right\}$ is a
	tight frame over $\mathbb{T}^{q}$: for any $f\in L^{2}\left( \mathbb{T}%
	^{q},\nu \right) $, let the needlet coefficients be given by%
	\begin{equation}\label{eqn:needcoeff}
	\betac=\int_{\mathbb{T}^{q}}f\left( \theta \right) \overline{\psi _{j,k}}%
	\left( \theta \right) \nu \left( d\theta \right).
	\end{equation}%
Then, it holds that $\sum_{j,k}\left\vert \beta _{j,k}\right\vert ^{2}=\left\Vert f\right\Vert
	_{L^{2}\left( \mathbb{T}^{q},\nu \right) }^{2}$. Therefore, the following reconstruction formula holds in the $L^{2}$-sense
	\begin{equation*}
	f\left( \theta \right) =\sum_{j\geq 0}\sum_{k=1}^{K_j}\beta _{j,k}\psi _{j,k}\left( \theta \right),\ \ \theta \in \mathbb{T}^{q}.
	\end{equation*}
\section{Proofs of the main results}\label{sec:approx}
\noindent Observe that, by Lemma \ref{lemmarepresentation}, $U_j\l(t\r)$ can be rewritten as a double Poisson integral of a needlet-based kernel, namely
	\begin{equation}
	\label{repofujasdoubleint}
	U_{j}\left( t\right) =\frac{1}{2}\int_{T} h_{j} \left( \theta _{1},\theta
	_{2}\right) \hat{N_{t}}\left( d\theta _{1}\right) \hat{N_{t}}%
	\left( d\theta _{2}\right),
	\end{equation}
where $h_{j}$ is given by \eqref{eqn:kerdef} and the integration domain $T$ is given by \eqref{eqn:T}. In the upcoming proof of Theorem \ref{maintheorem}, the strategy will be to apply Proposition \ref{wassboundfordouble} to the normalized version of the double integral representation of $U_j(t)$ given by \eqref{repofujasdoubleint}.
\begin{proof}[Proof of Theorem \ref{maintheorem}]
\noindent Starting with the first part of the statement and hence assuming that Condition \ref{cond1} holds, one can write, in view of \eqref{eqn:equivsum} and Lemma \ref{lemma:kerdue}, that
\begin{eqnarray*}
\norm{h_j}_{L^4\left(\mu_t^2 \right) }^{4} &=& 2R_t^2\int_{\Tq \times \Tq}\sum_{n_1,n_2,n_3,n_4 \in \Lambda_j}\prod_{i=1}^{4}b^2\left( B^{-j}\ell_{n_i}\right)s_{n_1}\left( \theta_1\right) \overline{s_{n_1}}\left( \theta_2\right)\overline{s_{n_2}}\left( \theta_1\right)\\&& s_{n_2}\left( \theta_2\right)s_{n_3}\left( \theta_1\right) \overline{s_{n_3}}\left( \theta_2\right) \overline{s_{n_4}}\left( \theta_1\right) s_{n_4}\left( \theta_2\right)f\left( \theta_1\right) \overline{f}\left( \theta_2\right)d\theta_1 d\theta_2 \\
&=& 2R_t^2\sum_{n_1,n_2,n_3,n_4 \in \Lambda_j}\prod_{i=1}^{4}b^2\left( B^{-j}\ell_{n_i}\right)a_{n_1 - n_2 + n_3 - n_4}\overline{a_{n_1 - n_2 + n_3 - n_4}}  \\
&\sim &  R_t^2\sum_{\ell_{n_1},\ell_{n_2}\ell_{n_3},\ell_{n_4} \in \sbra{B^{j-1},B^{j+1}}}c_{\ell_{n_1}}c_{\ell_{n_2}}c_{\ell_{n_3}}c_{\ell_{n_4}}
\prod_{i=1}^{4}b^2\left( B^{-j}\ell_{n_i}\right) \left(\ell_{n_1 - n_2 + n_3 - n_4} +1\right)^{-2\alpha}. 
\end{eqnarray*}
Fixing $\ell_{n_1},\ell_{n_2},\ell_{n_3}$ and summing over $\ell_{n_4}$ from zero to infinity yields
\begin{equation*}
\sum_{\ell_{n_4}=0}^{\infty}\left( \ell_{n_1 - n_2 + n_3 - n_4} +1\right)^{-2\alpha}=O\left(1\right)
\end{equation*}
as $\alpha > \frac{1}{2}$. Hence, 
\begin{equation*}
	\norm{h_j}_{L^4\left(\mu_t^2 \right) }^{4} 
	 \lesssim  R_t^2\sum_{\ell_{n_1},\ell_{n_2}\ell_{n_3}\in \sbra{B^{j-1},B^{j+1}}}1 = R_t^2 B^{3j}. 
\end{equation*}
In conclusion, using Lemma \ref{lemma:variancelower} yields
\begin{equation*}
\norm{\frac{h_j}{\sqrt{\mbox{\rm Var}\left(U_{j}(t)\right) }}}_{L^4\left(\mu_t^2 \right) }^{4} \lesssim O\left( R_t^{-2}B^{j}\right).
\end{equation*}
\noindent In order to compute $\norm{h_j \star_{2}^{1} h_j }_{L^2\left(\mu_t \right) }$, \eqref{eqn:equivsum} can be used once more to obtain 
\begin{eqnarray*}
&& \norm{h_j \star_{2}^{1} h_j }_{L^2\left(\mu_t \right) }^{2} =  \int_{T^3}h_{j}\left(\theta_1,\theta_2 \right)\overline{h_{j}}\left(\theta_1,\theta_2 \right)h_{j}\left(\theta_1,\theta_3 \right) \overline{h_{j}}\left(\theta_1,\theta_3 \right)\mu_t\left(d\theta_2\right)\overline{\mu_t}\left(d\theta_3\right)\mu_t\left(d\theta_1\right) \\
&& = 2R_t^3\int_{\Tq \times \Tq \times \Tq}\sum_{n_1,n_2,n_3,n_4 \in \Lambda_j}\prod_{i=1}^{4}b^2\left( B^{-j}\ell_{n_i}\right)s_{n_1}\left( \theta_1\right) \overline{s_{n_1}}\left( \theta_2\right) \overline{s_{n_2}}\left( \theta_1\right) s_{n_2}\left( \theta_2\right)s_{n_3}\left( \theta_1\right) \overline{s_{n_3}}\left( \theta_3\right) \overline{s_{n_4}}\left( \theta_1\right) \\
&& \qquad\qquad\qquad\qquad\qquad\qquad\qquad\qquad\qquad\qquad\qquad\qquad\qquad\qquad  s_{n_4}\left( \theta_3\right) f\left( \theta_2\right) \overline{f}\left( \theta_3\right) f\left( \theta_1\right)  d\theta_2 d\theta_3 d\theta_1 \\
&& = 2R_t^3\sum_{n_1,n_2,n_3,n_4 \in \Lambda_j}\prod_{i=1}^{4}b^2\left( B^{-j}\ell_{n_i}\right)a_{n_1 - n_2 + n_3 - n_4}\overline{a_{n_3 - n_4}}a_{n_1 - n_2} \\
&& \sim  R_t^3\sum_{\ell_{n_1},\ell_{n_2},\ell_{n_3},\ell_{n_4} \in \sbra{B^{j-1},B^{j+1}}}c_{\ell_{n_1}}c_{\ell_{n_2}}c_{\ell_{n_3}}c_{\ell_{n_4}}\prod_{i=1}^{4}b^2\left( B^{-j}\ell_{n_i}\right) \left(\ell_{ n_1 - n_2 + n_3 - n_4} +1\right)^{-\alpha} \\
&& \qquad\qquad\qquad\qquad\qquad\qquad\qquad\qquad\qquad\qquad\qquad\qquad\qquad\qquad\quad \left(\ell_{ n_3 - n_4}+1\right)^{-\alpha_1}\left(\ell_{ n_1 - n_2} +1\right)^{-\alpha} \\
&& \lesssim  R_t^3\sum_{\ell_{n_1},\ell_{n_2},\ell_{n_3},\ell_{n_4} \in \sbra{B^{j-1},B^{j+1}}}\left(\ell_{ n_3 - n_4} +1\right)^{-\alpha}\left(\ell_{n_1 - n_2}+1\right)^{-\alpha}.  
\end{eqnarray*}
In the case where $\alpha>1$, fixing $n_1,n_3$ and summing over $n_2,n_4$ yields $\norm{h_j \star_{2}^{1} h_j }_{L^2\left(\mu_t \right) }^{2}
	 \lesssim  R_t^3B^{2j}$. On the other hand, if $\frac{1}{2}<\alpha<1$, we use a Riemann sum argument to get $\norm{h_j \star_{2}^{1} h_j }_{L^2\left(\mu_t \right) }^{2}
\lesssim  R_t^3B^{2j\left(2-\alpha\right)}$. In conclusion, we have
\begin{equation*}
\norm{\left( \frac{h_j}{\sqrt{\mbox{\rm Var}\left(U_{j}(t)\right) }}\right) \star_{2}^{1} \left( \frac{h_j}{\sqrt{\mbox{\rm Var}\left(U_{j}(t)\right) }}\right) }_{L^2\left(\mu_t \right) }^{2} \lesssim \begin{cases} R_t^{-1} & \text{if } \alpha >1, \\ R_t^{-1}B^{2j\left(1-\alpha\right)} & \text{if } \frac{1}{2}<\alpha <1. \end{cases}
\end{equation*}
Finally, using \eqref{eqn:equivsum} again, it holds that 
\begin{eqnarray*}
\norm{h_j \star_{1}^{1} h_j }_{L^2\left(\mu_t \right) }^{2} &=& 
\int_{T^4}h_{j}\left(\theta_1,\theta_3 \right)\overline{h_{j}}\left(\theta_2,\theta_3 \right)h_{j}\left(\theta_1,\theta_4 \right)\overline{h_{j}}\left(\theta_2,\theta_4 \right)\mu_t\left(d\theta_3\right)\overline{\mu_t}\left(d\theta_4\right)\mu_t\left(d\theta_1\right)\overline{\mu_t}\left(d\theta_2\right) \\
&=& 2R_t^4\int_{\Tq \times \Tq \times\Tq \times\Tq }\sum_{n_1,n_2,n_3,n_4 \in \Lambda_j}\prod_{i=1}^{4}b^2\left( B^{-j}\ell_{n_i}\right)s_{n_1}\left( \theta_1\right) \overline{s_{n_1}}\left( \theta_3\right)\\&&\overline{s_{n_2}}\left( \theta_2\right) s_{n_2}\left( \theta_3\right)s_{n_3}\left( \theta_1\right) \overline{s_{n_3}}\left( \theta_4\right) \overline{s_{n_4}}\left( \theta_2\right) s_{n_4}\left( \theta_4\right)f\left( \theta_1\right) \overline{f}\left( \theta_2\right) f\left( \theta_3\right) \overline{f}\left( \theta_4\right)  d\theta_1 d\theta_2 d\theta_3 d\theta_4 \\
&=& 2R_t^4\sum_{n_1,n_2,n_3,n_4 \in \Lambda_j}\prod_{i=1}^{4}b^2\left( B^{-j}\ell_{n_i}\right)a_{n_1 + n_3}\overline{a_{n_2+ n_4}}a_{n_2 - n_1}\overline{a_{n_3- n_4}}\\
& \lesssim & R_t^4\sum_{\ell_{n_1},\ell_{n_2},\ell_{n_3},\ell_{n_4} \in \sbra{B^{j-1},B^{j+1}}} c_{\ell_{n_1}}c_{\ell_{n_2}}c_{\ell_{n_3}}c_{\ell_{n_4}} \\
&& \qquad\qquad\qquad  \left(\ell_{n_1 + n_3}+1\right)^{-\alpha}\left(\ell_{ n_2+ n_4} +1\right)^{-\alpha}\left(\ell_{ n_2 - n_1} +1\right)^{-\alpha}\left(\ell_{ n_3- n_4}+1\right)^{-\alpha}  .
\end{eqnarray*}
In the case where $\alpha >1$, using the same summability argument as before, one gets
\begin{equation*}
	\norm{h_j \star_{1}^{1}  h_j }_{L^2\left(\mu_t \right) }^{2} \lesssim  R_t^4\sum_{\ell_{n_1},\ell_{n_2},\ell_{n_3},\ell_{n_4} \in \sbra{B^{j-1},B^{j+1}}}\left( \ell_{n_2+ n_4} +1\right)^{-\alpha}\left(\ell_{ n_2 - n_1} +1\right)^{-\alpha}\left(\ell_{ n_3- n_4} +1\right)^{-\alpha}  \lesssim R_t^4 B^j.  
\end{equation*}
In the case where $\frac{1}{2}\leq \alpha \leq 1$, one can use a Riemann sum argument once again to get 
\begin{eqnarray*}
	\norm{h_j \star_{1}^{1} h_j }_{L^2\left(\mu_t \right) }^{2} &\lesssim &  R_t^4\sum_{\ell_{n_1},\ell_{n_2},\ell_{n_3},\ell_{n_4} \in \sbra{B^{j-1},B^{j+1}}} \left(\ell_{ n_2+ n_4}+1\right)^{-\alpha}\left(\ell_{n_2 - n_1 } +1\right)^{-\alpha}\left(\ell_{ n_3- n_4 }+1\right)^{-\alpha}  \\
	&\lesssim & R_t^4 B^{4j\left(1-\alpha\right)}.
\end{eqnarray*}
Summing up yields
\begin{equation*}
\norm{\left( \frac{h_j}{\sqrt{\mbox{\rm Var}\left(U_{j}(t)\right) }}\right) \star_{1}^{1} \left( \frac{h_j}{\sqrt{\mbox{\rm Var}\left(U_{j}(t)\right) }}\right) }_{L^2\left(\mu_t \right) }^{2} \lesssim \begin{cases}  B^{-j} & \text{if } \alpha >1, \\ B^{2j\left(1-2\alpha\right)} & \text{if } \frac{1}{2}<\alpha <1. \end{cases}
\end{equation*}
Applying Proposition \ref{wassboundfordouble} concludes the first part of the proof.
\\~\\
For the second part of the statement, assume that Condition \ref{cond2} holds and observe that for any function $$g\bra{n}=\prod_{m=1}^qg_m\bra{n_{m}},$$ there exists a constant $C>0$ such that
	\begin{eqnarray}
	\sum_{n\in\Lambda_j} b^2\bra{B^{-j}\ell_n}g\bra{n} &\leq & \prod_{m=1}^q\sum_{n_{m}\in\left[B^{j-1},B^{j+1}\right]} \bra{\sup_{x\in\left[B^{j-1},B^{j+1}\right]} b\bra{x}} g_m\bra{n_{m}} \nonumber\\
	&\leq & C\prod_{m=1}^q\sum_{n_{m}\in\left[B^{j-1},B^{j+1}\right]}  g_m\bra{n_{m}}.\label{eqn:bbound}
	\end{eqnarray}
Using \eqref{eqn:bbound} and following a similar procedure to the one used to prove the first part of the statement, one obtains
	\begin{eqnarray*}
		\norm{h_j}_{L^4\left(\mu_t^2 \right) }^{4} 
		&=& 2R_t^2\sum_{n_1,n_2,n_3,n_4 \in \Lambda_j}\prod_{i=1}^{4}b^2\left( B^{-j}\ell_{n_i}\right)a_{n_1 - n_2 + n_3 - n_4}\overline{a_{n_1 - n_2 + n_3 - n_4}} \\
		& \sim & R_t^2\prod_{m=1}^q \bra{\vert n_{1,m} - n_{2,m} +n_{3,m}- n_{4,m}\vert +1}^{-2\alpha}. 
	\end{eqnarray*}
	For any $m=1,\ldots,q$, fixing $n_{1,m},n_{2,m},n_{3,m}$ and sum over $n_{4,m}$ from zero to infinity yields
	\begin{equation*}
	\sum_{n_{4,m}=0}^{\infty}\left( \vert n_{1,m} - n_{2,m} + n_{3,m} - n_{4,m} +1\right)^{-2\alpha}=O\left(1\right)
	\end{equation*}
	as $\alpha > \frac{1}{2}$. Hence, 
	\begin{equation*}
	\norm{h_j}_{L^4\left(\mu_t^2 \right) }^{4} 
	\lesssim  R_t^2\sum_{\ell_{n_1},\ell_{n_2}\ell_{n_3}\in \sbra{B^{j-1},B^{j+1}}}1 = R_t^2 B^{3jq}. 
	\end{equation*}
	In conclusion, using Lemma \ref{lemma:variancelower}
	\begin{equation*}
	\norm{\frac{h_j}{\sqrt{\mbox{\rm Var}\left(U_{j}(t)\right) }}}_{L^4\left(\mu_t^2 \right) }^{4} \lesssim  R_t^{-2}B^{jq}.
	\end{equation*}
	A similar procedure is used also to establish an upper bound for $\norm{h_j \star_{2}^{1} h_j }_{L^2\left(\mu_t \right) }$ and $\norm{h_j \star_{1}^{1} h_j }_{L^2\left(\mu_t \right) }$, leading to
	\begin{equation*}
	\norm{\left( \frac{h_j}{\sqrt{\mbox{\rm Var}\left(U_{j}(t)\right) }}\right) \star_{2}^{1} \left( \frac{h_j}{\sqrt{\mbox{\rm Var}\left(U_{j}(t)\right) }}\right) }_{L^2\left(\mu_t \right) }^{2} \lesssim \begin{cases} R_t^{-1} & \text{if } \alpha >1, \\ R_t^{-1}B^{2j\left(1-\alpha\right)} & \text{if } \frac{1}{2}<\alpha <1 \end{cases}
	\end{equation*}
and
	\begin{equation*}
	\norm{\left( \frac{h_j}{\sqrt{\mbox{\rm Var}\left(U_{j}(t)\right) }}\right) \star_{1}^{1} \left( \frac{h_j}{\sqrt{\mbox{\rm Var}\left(U_{j}(t)\right) }}\right) }_{L^2\left(\mu_t \right) }^{2} \lesssim \begin{cases}  B^{-j} & \text{if } \alpha >1, \\ B^{2j\left(1-2\alpha\right)} & \text{if } \frac{1}{2}<\alpha <1. \end{cases}
	\end{equation*}
Applying Proposition \ref{wassboundfordouble} concludes the proof.
\end{proof}

\section{Auxiliary results}\label{sec:proof}
\noindent The following lemma shows that the estimator $U_{j}\left( t\right)$ defined in \eqref{mainestimator} can be represented as a double Poisson integral.
\begin{lemma}	\label{lemmarepresentation}Let $U_{j}\left( t\right) $ be given by \eqref{mainestimator}. Then, it holds that 
	\begin{equation}\label{eqn:kernel}
	U_{j}\left( t\right) =\frac{1}{2}\int_{T} h_{j} \left( \theta _{1},\theta
	_{2}\right) \hat{N_{t}}\left( d\theta _{1}\right) \hat{N_{t}}%
	\left( d\theta _{2}\right),
	\end{equation}
	where the symmetric function $h_{j} \in L^2\left(\mu_t^2 \right) $ is given by 
	\begin{eqnarray}
	h_{j}\left( \theta _{1},\theta _{2}\right)  &:=&\sum_{k=1}^{K_j}\left[ \psi
	_{j,k}\left( \theta _{1}\right) \psi _{j,k}\left( \theta _{2}\right) \ind{\mathbb{T}^q_{1}\times \mathbb{T}^q_{1}}
	\left( \theta _{1},\theta_{2}\right) -\psi _{j,k}\left( \theta _{1}\right) \psi _{j,k}\left( \theta
	_{2}\right) \ind{\mathbb{T}^q_{1}\times \mathbb{T}^q_{2}}\left(
	\theta _{1},\theta _{2}\right) \right.  \nonumber \\
	&& \left. \qquad\qquad - \psi _{j,k}\left( \theta _{1}\right) \psi _{j,k}\left( \theta
	_{2}\right) \ind{\mathbb{T}^q_{2}\times \mathbb{T}^q_{1}}\left(
	\theta _{1},\theta _{2}\right)+\psi _{j,k}\left( \theta _{1}\right) \psi
	_{j,k}\left( \theta _{2}\right) \ind{\mathbb{T}^q_{2}\times \mathbb{T}^q_{2}}\left( \theta _{1},\theta _{2}\right) \right]
	\label{eqn:kerdef}
\end{eqnarray}
and where the integration domain $T$ is defined as
\begin{equation}\label{eqn:T}
T := \left( \mathbb{T}^q_{1}\times \mathbb{T}^q_{1}\right) \cup \left( \mathbb{T}^q_{1}\times \mathbb{T}^q_{2}\right) \cup \left( \mathbb{T}^q_{2}\times \mathbb{T}^q_{1}\right) \cup \left( \mathbb{T}^q_{2}\times \mathbb{T}^q_{2}\right).
\end{equation}
\end{lemma}
\begin{proof}
Centering the measures $N_{t}$ appearing in the expression of $U_{j}\left( t\right) $ given by \eqref{mainestimator} yields
	\begin{eqnarray*}
		U_{j}\left( t\right)  &=& \sum_{k=1}^{K_j}\left[ \left( \int_{\mathbb{T}^q_{1}}\psi
		_{j,k}\left( \theta _{1}\right) \hat{N_{t}}\left( d\theta _{1}\right) -\int_{\mathbb{T}^q_{2}}\psi _{j,k}\left( \theta _{1}\right) \hat{	N_{t}}\left( d\theta _{1}\right) \right) ^{2} \right.  \\
		&& \left. \qquad\qquad \qquad \qquad \qquad- \int_{\mathbb{T}^q_{1}}\psi _{j,k}^{2}\left( \theta _{1}\right) 
		\hat{N_{t}}\left( d\theta _{1}\right) -\int_{\mathbb{T}^q_{2}}\psi
		_{j,k}^{2}\left( \theta _{1}\right) \hat{N_{t}}\left( d\theta
		_{1}\right) -2\left\Vert \psi _{j,k}\right\Vert _{L^{2}\left( \mu _{t}\right)
		}^{2}\right].
	\end{eqnarray*}
From expanding the square and using the product formula for Poisson multiple integrals stated in Proposition \ref{prodformula}, it follows that
	\begin{eqnarray*}
		U_{j}\left( t\right)  &=&\sum_{k=1}^{K_j}\left[ \int_{\mathbb{T}^q_{1}\times \mathbb{T}^q_{1}}
		\psi _{j,k}\left( \theta _{1}\right) \psi _{j,k}\left( \theta _{2}\right)
	 \hat{N_{t}}\left( d\theta _{1}\right) \hat{N_{t}}\left(
		d\theta _{2}\right) 
		+\int_{\mathbb{T}^q_{2}\times \mathbb{T}^q_{2}}\left( \psi _{j,k}\left( \theta _{1}\right) \psi _{j,k}\left( \theta
		_{2}\right) \right) \hat{N_{t}}\left( d\theta _{1}\right) \hat{%
			N_{t}}\left( d\theta _{2}\right) \right.  \\
			&&\left. 	-\int_{\mathbb{T}^q_{1}\times \mathbb{T}^q_{2}}\psi _{j,k}\left( \theta
		_{1}\right) \psi _{j,k}\left( \theta _{2}\right) \hat{N_{t}}\left(
		d\theta _{1}\right) \hat{N_{t}}\left( d\theta _{2}\right) 
		-\int_{\mathbb{T}^q_{2}\times \mathbb{T}^q_{1}}\psi_{j,k}\left( \theta
		_{1}\right) \psi _{j,k}\left( \theta _{2}\right) \hat{N_{t}}\left(
		d\theta _{1}\right) \hat{N_{t}}\left( d\theta _{2}\right) \right] \\
	&=&\frac{1}{2}\int_{T}h_{j}\left( \theta _{1},\theta _{2}\right) \hat{%
		N_{t}}\left( d\theta _{1}\right) \hat{N_{t}}\left( d\theta _{2}\right),
	\end{eqnarray*}
where $h_{j}\left( \cdot ,\cdot \right) $ is the kernel given by \eqref{eqn:kerdef} and $T$ the domain defined in \eqref{eqn:T}, as claimed.
\end{proof}
\noindent The following result provides an alternate expression of the kernel $h_j$ be defined in \eqref{eqn:kerdef} used to reduce the complexity of the domain $T$ defined in \eqref{eqn:T}.
\begin{lemma}\label{lemma:kerdue}
	Let $h_j$ be given by \eqref{eqn:kerdef}. For all $a,b=1,2$ and for $\theta_1,\theta_2 \in \mathbb{T}^q$, define
	\begin{equation}\label{eqn:kerpart}
	h_{ab,j}\bra{\theta_1,\theta_2}=\sum_{n\in\Lambda_j} b^2\bra{\barg}s_n\bra{\theta_1}\cc{s_n\bra{\theta_2}}\ind{\mathbb{T}^q_a\times \mathbb{T}^q_b}\bra{\theta_1,\theta_2}.
	\end{equation}
	Then, it holds 
	\begin{equation}\label{eqn:kerdecomp}
	h_j\bra{\theta_1,\theta_2}=\sum_{a,b=1}^2\bra{-1}^{a+b}h_{ab,j}\bra{\theta_1,\theta_2}.
	\end{equation}
\end{lemma}
\begin{proof}
Observe that, by using \eqref{eqn:needlet}, one can write
\begin{eqnarray*}
\sumk \needlet{\theta_1}\cc{\needlet{\theta_2}}\ind{\mathbb{T}^q_a\times\mathbb{T}^q_b}\bra{\theta_1,\theta_2}&=&\sum_{n_1,n_2 \in \Lambda_j}\bfun{\frac{\ell_{n_1}}{B^j}}\bfun{\frac{\ell_{n_2}}{B^j}}\ind{\mathbb{T}^q_a\times\mathbb{T}^q_b}\bra{\theta_1,\theta_2}
s_{n_1}\bra{\theta_1}\cc{s_{n_2}\bra{\theta_2}} \\
&& \qquad\qquad\qquad\qquad\qquad\qquad  \times \sumk \cubew \cc{s_{n_1}\bra{\cubep}}s_{n_2}\bra{\cubep}.
\end{eqnarray*}
Using the fact that $\sumk \cubew \cc{s_{n_1}\bra{\cubep}}s_{n_2}\bra{\cubep} = \int_{\Tq} \cc{s_{n_1}\bra{\theta}}s_{n_2}\bra{\theta}\nu\left(d\theta\right)$ yields
\begin{eqnarray*}
\sumk \needlet{\theta_1}\cc{\needlet{\theta_2}}\ind{\mathbb{T}^q_a\times\mathbb{T}^q_b}\bra{\theta_1,\theta_2}&=&\sum_{n_1,n_2 \in \Lambda_j}\bfun{\frac{\ell_{n_1}}{B^j}}\bfun{\frac{\ell_{n_2}}{B^j}}\ind{\mathbb{T}^q_a\times\mathbb{T}^q_b}\bra{\theta_1,\theta_2}
s_{n_1}\bra{\theta_1}\cc{s_{n_2}\bra{\theta_2}} \\
&& \qquad\qquad\qquad\qquad\qquad\qquad  \int_{\Tq} \cc{s_{n_1}\bra{\theta}}s_{n_2}\bra{\theta}\nu\left(d\theta\right)\\
&=&\sum_{n \in \Lambda_j}b^2\bra{\frac{\ell_{n}}{B^j}}\ind{\mathbb{T}^q_a\times\mathbb{T}^q_b}\bra{\theta_1,\theta_2}s_n\bra{\theta_1-\theta_2},
\end{eqnarray*}
where the last equality comes from the orthogonality property \eqref{eqn:ortho}. Recalling the definition of $h_j$, given by \eqref{eqn:kerdef}, concludes the proof.
\end{proof}
\noindent The following lemma provides an explicit expression for the variance of $U_j(t)$, which will be crucial in deriving lower bounds for said variance in Lemma \ref{lemma:variancelower}.
\begin{lemma}\label{lemma:umoments}
Let $U_j\bra{t}$ be given by \eqref{mainestimator}. It holds that $\Ex\bra{U_j\bra{t}}=0$ and
\begin{equation*}
\Ex\bra{U^2_j\bra{t}} = 8R_t^2\sum_{n_1,n_2 \in \Lambda_{j}}b^2\left( B^{-j}\ell_{n_1}\right)b^2\left( B^{-j}\ell_{n_2}\right)\vert a_{n_1-n_2} \vert^2.
\end{equation*}
\end{lemma}
\begin{proof}
Using the representation of $U_j\bra{t}$ given by Lemma \ref{lemmarepresentation}, it is easily seen that $\Ex\bra{U_j\bra{t}}=0$. On the other hand, the isometry property of Poisson multiple integrals (see Proposition \ref{isometryprop}) yields $\E\left( U_{j}(t)^2\right) =  2\norm{h_j}_{L^2\left( \mu_t^2\right) }^{2}$, where, using Lemma \ref{lemma:kerdue}, it holds that
\begin{equation*}
2\norm{h_j}_{L^2\left( \mu_t^2\right) }^{2}= 2\int_{T}\bra{\sum_{a,b=1}^2 \bra{-1}^{a+b}h_{ab,j}\bra{\theta_1,\theta_2}}^2\mu_t\bra{d\theta_1} \cc{\mu_t\bra{d\theta_2}} := V_1+V_2, 
\end{equation*}
with
\begin{equation*}
V_1 = 2\int_T \sum_{a,b=1}^2 h_{ab,j}^2\l(\theta_1,\theta_2\r)\mu_t\l(d\theta_1\r) \cc{\mu_t}\l(d\theta_2\r)
\end{equation*}
and 
\begin{eqnarray*}
V_2 &=& 4\int_T h_{11,j}\l(\theta_1,\theta_2\r)h_{22,j}\l(\theta_1,\theta_2\r)\mu_t\l(d\theta_1\r) \cc{\mu_t}\l(d\theta_2\r)+ 4\int_T  h_{12,j}\l(\theta_1,\theta_2\r)h_{21,j}\l(\theta_1,\theta_2\r)\mu_t\l(d\theta_1\r) \cc{\mu_t}\l(d\theta_2\r)\\
&& \qquad\qquad\qquad\qquad\quad + 4\int_T \sum_{b\neq a, a=1}^2 h_{aa,j}\l(\theta_1,\theta_2\r)\bra{h_{ba,j}\l(\theta_1,\theta_2\r)+h_{ab,j}\l(\theta_1,\theta_2\r)}\mu_t\l(d\theta_1\r) \cc{\mu_t}\l(d\theta_2\r).
\end{eqnarray*}
Observe that $V_2 = 0$ as each $\theta_i$, $i=1,2$, can only belong to one of the disjoint tori $\Tq_1$ or $\Tq_2$. On the other hand, using the Fourier expansion of the density function given by $$f\left( \theta \right) = \sum_{n \in \mathbb{Z}^q}a_n s_n\left(\theta \right),$$ we get 
\begin{eqnarray*}
	V_1 &=& 8R_t^2\sum_{n_1,n_2 \in \Lambda_{j}}\sum_{n_3,n_4 \in \mathbb{Z}}b^2\left( \frac{\ell_{n_1}}{B^{j}}\right)b^2\left( \frac{\ell_{n_2}}{B^{j}}\right)a_{n_3}\overline{a_{n_4}} \int_{\mathbb{T}^q}s_{n_1}\left( \theta_1\right)\overline{s_{n_2}}\left( \theta_1\right)s_{n_3}\left( \theta_1\right)  d\theta_1  \\
	&& \qquad\qquad\qquad\qquad\qquad\qquad\qquad\qquad\qquad\qquad\qquad\qquad\qquad\qquad \int_{\mathbb{T}^q}\overline{s_{n_1}}\left( \theta_2\right) s_{n_2}\left( \theta_2\right)\overline{s_{n_4}}\left(\theta_2 \right)  d\theta_2 \\
	&=& 8R_t^2\sum_{n_1,n_2 \in \Lambda_{j}}\sum_{n_3,n_4 \in \mathbb{Z}}b^2\left( \frac{\vert n_1\vert}{B^{j}}\right)b^2\left( \frac{\vert n_2\vert}{B^{j}}\right)a_{n_3}\overline{a_{n_4}}\delta_{n_1 - n_2}^{n_3} \delta_{n_1 - n_2}^{n_4} \\
	&=& 8R_t^2\sum_{n_1,n_2 \in \Lambda_{j}}b^2\left( \frac{\ell_{n_1}}{B^{j}}\right)b^2\left( \frac{\ell_{n_2}}{B^{j}}\right)\vert a_{n_1-n_2} \vert^2, 
\end{eqnarray*}
as claimed.
\end{proof}
\noindent Finally, under the assumptions stated in Conditions \ref{cond1} and \ref{cond2}, the following result provides lower bounds for the variance of $U_j(t)$.
\begin{lemma}\label{lemma:variancelower}
Under Condition \ref{cond1}, it holds that
\begin{equation}\label{eqn:variancelow1}
\E\left( U_{j}(t)^2\right) \gtrsim  R_t^2 B^j.
\end{equation}
Under Condition \ref{cond2}, it holds that
\begin{equation}\label{eqn:variancelow2}
 \E\left( U_{j}(t)^2\right) \gtrsim  R_t^2 B^{qj}.
\end{equation} 
\end{lemma}
\begin{proof}
Recalling that under Condition \ref{cond1}, the Fourier coefficients satisfy $a_{n} \sim \left(\ell_n+1\right)^{-\alpha}$, with $\alpha > \frac{1}{2}$ in order to guaranty that $\sum_{n \in \mathbb{Z}^q}\vert a_n \vert^2 < +\infty$, yields
\begin{eqnarray*}
\E \left( U_{j} (t)^2 \right)   &\sim &  R_t^2\sum_{n_1,n_2 \in \Lambda_{j}}b^2\left(B^{-j}\ell_{n_1}\right)b^2\left( B^{-j}\ell_{n_2}\right)\left( \ell_{n_1-n_2} +1\right)^{-2\alpha}\\
& \gtrsim & R_t^2\sum_{n_1,n_2 \in \Lambda'_{j}}b^2\left( B^{-j}\ell_{ n_1}\right)b^2\left( B^{-j}\ell_{n_2}\right)\left(\ell_{ n_1-n_2} +1\right)^{-2\alpha},
\end{eqnarray*}
where $\Lambda^\prime_{j} = \left\lbrace n \in \mathbb{Z}^q \colon \ell_n  \in \left[\left(B^\prime\right)^{j-1},\left(B^\prime\right)^{j+1} \right]  \right\rbrace$ with $B^\prime <B$ such that $b\left(B^{-j}n\right)$ is bounded away from zero. Now, observe that, for any given real valued function $g$, there exists a set of coefficients $\bbra{c_{\ell_n}}$ so that 
\begin{equation}\label{eqn:equivsum}
\sum_{n\in \Lambda_j} g\bra{\ell_n}= \sum_{\ell_n\in \sbra{B^{j-1},B^{j+1}}} c_{\ell_n} g\bra{\ell_n}.
\end{equation} 
Loosely speaking, $0\leq c_{\ell_n}<\infty$ denotes the number of possible combinations of components of different $n \in \Lambda_j$ corresponding to the same $\ell_n$.
Dropping the terms for which $n_1 \neq n_2$ in the above sum, we get $ \E\left( U_{j}(t)^2\right) \gtrsim  R_t^2 B^j$, as claimed. 
\\~\\
Under Condition \ref{cond2}, the Fourier coefficients are such that $a_{n} \sim \prod_{m=1}^q\left(\vert n_{\bra{m}}\vert+1\right)^{-\alpha}$, with $\alpha > \frac{1}{2}$. In this case, one gets that
\begin{equation*}
	\E \left( U_{j} (t)^2\right)   \gtrsim   R_t^2\sum_{n_1,n_2 \in \Lambda'_{j}}b^2\left( B^{-j}\ell_{ n_1}\right)b^2\left( B^{-j}\ell_{n_2}\right)\prod_{m=1}^q\left( \vert n_1-n_2\vert +1\right)^{-2\alpha}.
\end{equation*}
Considering each component separately, one obtains that $\E \left( U_{j} (t)^2\right)   \gtrsim   R_t^2B^{qj}$, as claimed.
\end{proof}
\section*{Acknowledgements}
\noindent The authors wish to thank Domenico Marinucci for many useful remarks on an earlier version of this paper as well as for interesting discussions on this topic.

\printbibliography
\end{document}